\documentclass[11pt, a4paper]{amsart}
\usepackage{amssymb}
\usepackage{bm}
\usepackage{amssymb}
\usepackage{color}

\newtheorem{lemm}{Lemma}[section]
\newtheorem{coro}{Corollary}[section]
\newtheorem{prop}{Proposition}[section]
\newtheorem{theo}{Theorem}[section]
\theoremstyle{definition}

\newtheorem{rema}{Remark}[section]
\numberwithin{equation}{section}
\newtheorem*{ack}{Acknowledgements}
\def\Bbb{\mathbb}
\def\H{\mathbb H}
\def\hh{\H}
\def\S{\mathbb S}
\def\ms{\S}

\def\ep{\varepsilon}

\def\a{\alpha}

\def\fint{\operatorname {--\!\!\!\!\!\int\!\!\!\!\!--}}

\def\rn1{\Bbb R^{N+1}}
\def\rn{\Bbb R^{N}}

\def\a*{\alpha^*_\ep}

\def\wep{w^\ep}

\def\R{\mathbb R}
\def\rr{\R}

\newcommand{\lp}{\left(}
\newcommand{\rp}{\right)}
\def\N{\mathbb N}
\def\a{\mathfrak q}

\begin{document}

\title[Nonlocal diffusion on manifolds]{A nonlocal diffusion problem on  manifolds}

\author[Catherine Bandle, Maria del Mar Gonzalez, Marco A. Fontelos \& Noemi Wolanski]
{Catherine Bandle, Maria del Mar Gonzalez,\\ Marco A. Fontelos and Noemi Wolanski}

\address{ Noemi Wolanski
\hfill\break\indent Departamento  de Matem\'atica, FCEyN-UBA and IMAS, CONICET
\hfill\break\indent Ciudad Universitaria, Pabellon I (1428) Buenos Aires, Argentina.}
\email{{\tt
wolanski@dm.uba.ar}\hfill\break\indent {\em Web-page:} {\tt
http://mate.dm.uba.ar/$\sim$wolanski}}

\address{Catherine Bandle \hfill\break\indent
Department of Mathematics \hfill\break\indent
Spiegelgasse 1
CH-4051 Basel (Switzerland)}
\email{{\tt c.bandle@gmx.ch}}

\address{Marco A. Fontelos
\hfill\break\indent Instituto de Ciencias Matem\'aticas
\hfill\break\indent C/ Nicol\'as Cabrera 13-15, Campus de Cantoblanco, UAM, 28049 Madrid SPAIN}
\email{{\tt marco.fontelos@icmat.es
}\hfill\break\indent {\em Web-page:} {\tt
http://www.icmat.es/mafontelos}}

\address{Mar\'ia del Mar Gonz\'alez
\hfill\break\indent ETSEIB - Departamento de Matem\'aticas, Universidad Polit\'ecnica de Catalunya
\hfill\break\indent Av. Diagonal 647, Barcelona 08028, SPAIN}
\email{{\tt mar.gonzalez@upc.edu}\hfill\break\indent {\em Web-page:} {\tt
http://www.pagines.ma1.upc.edu/$\sim$mgonzalez/}}

\thanks{M. Fontelos is supported grant MTM2014-57158-R from the Spanish government. M.d.M. Gonzalez is part of the Catalan research group 2014 SGR 1083, and is supported by Spanish government grant MTM2014-52402-C3-1-P.
N. Wolanski is supported by ANPCyT PICT No. 2012-0153 and CONICET PIP 625 Res. 960/12.
N. Wolanski is a member of CONICET.}


\begin{abstract}
In this paper we study a nonlocal diffusion problem on
a manifold. These kind of equations can model diffusions when there are long range effects and have been widely studied in Euclidean space.
 We first prove existence and uniqueness of solutions and a comparison principle. Then, for a convenient rescaling we prove that the operator under consideration converges to a multiple of the usual Heat-Beltrami operator on the manifold. Next, we look at the long time behavior on compact manifolds by studying the spectral properties of the operator. Finally, for the model case of hyperbolic space we study the long time asymptotics and find a different and interesting behavior.
 \end{abstract}

\maketitle

\section{Introduction.}
In this paper we present  a nonlocal diffusion equation set on a manifold and study the properties of its solutions. More precisely, let $M$ be an $N-$dimensional, Riemannian manifold without boundary. We assume that there exists a family of isometries $\tau_x:M\to M$ such that $\tau_x(x)=O$ (were we have denoted by $O$ a fixed point in $M$). Let $d\mu$ be a measure on $M$,  invariant under the family of isometries $\tau_x$. Denote by $s_{xy}$ the geodesic distance between the points $x$ and $y$. For $J: \R_{\ge0}\to \R_{\ge0}$ normalized by
\begin{equation}\label{mass=1}
\int_M J(s_{xO})\, d\mu_x=1,
\end{equation}
 we define the operator
\begin{equation*}
Lu(x):=\int_M J(s_{xy})u(y)\,d\mu_y-u(x),
\end{equation*}
and consider the corresponding Cauchy problem
\begin{equation}\label{P}
 \begin{cases}
  u_t-Lu=0\qquad&\mbox{in }M\times(0,T],\\
u(x,0)=u_0(x)\qquad&\mbox{in }M.
 \end{cases}
\end{equation}
At some points of the article we  will assume further that $J\in C_0^\infty(\R_{\ge0})$.\\

These kind of diffusion problems have been widely studied in Euclidean space as they model dispersion when there are long range effects (see for instance, \cite{BZ, CF,Fi}).
Classical diffusion problems can be seen as infinitesimal limits of these nonlocal diffusion models. This is a very well studied fact from a probabilistic point of view (see, for instance \cite{D,Bertoin}) and has also a PDE counterpart \cite{CERW,CER}.

Other nonlocal models correspond to singular kernels and have also been  widely studied (we cite here \cite{CS,Vazquez:nonlocal}, for instance, but the literature is huge). This study is being extended to the manifold case, especially on non-compact manifolds such as hyperbolic space $\mathbb H^N$ (\cite{Banica-Gonzalez-Saez,Gonzalez-Saez-Sire}). Nevertheless, we will not consider singular kernels in this article.

In Euclidean space, the relation between jump diffusions as in \eqref{P} and the classical heat equation appears in two different asymptotic limits. On one hand, as we have referred to above, in the infinitesimal limit as jumps go to zero, on the other, as the large time asymptotics. In fact, when stated in $\R^N$ with $u_0\in L^1\cap L^\infty$,    the solution of \eqref{P} decays as $t\to\infty$ with the same rate as the one of the heat equation. Moreover, after properly rescaled, the limit profile is  the same as the one of the infinitesimal heat equation with diffusivity $\a=\frac1{2N}\int J(\xi)|\xi|^2\,d\xi$ and the same initial condition.

These two facts can be seen as the two sides of a coin. They are both a consequence of the fact that for $v\in C_0^\infty(\R^N)$ and $J_\ep(\xi)=\ep^{-N}J(\ep^{-1}\xi)$,
\[
\lim_{\ep\to0}\frac1{\ep^2}\int J_\ep(x-y)\big(v(y)-v(x)\big)\,dy= \a \Delta v(x).
\]
See, for instance, \cite{TW-critical} for this discussion.

In this paper, after the study of the well posedness of problem \eqref{P}, we address these kind of questions. We  are interested in studying if and how the geometry of $M$ influences the asymptotics of the nonlocal problem \eqref{P} both in the infinitesimal limit and in the large time behavior. In particular, in order to understand the influence of geometry at infinity we will concentrate on hyperbolic space. Note that (local) diffusions on hyperbolic space and on general noncompact manifolds have been considered by many authors (without being exhaustive, we cite \cite{Grigoryan:book,Saloff-Coste,Davies-Mandouvalos,Vazquez:hyperbolic,Bonforte-Grillo-Vazquez}).

Our first result on the infinitesimal limit states that, for radially symmetric manifolds $M$ the infinitesimal limit is the Heat-Beltrami problem (Theorem \ref{thm-radial}).

On the other hand, as can be expected, there is a big difference in the time asymptotics between compact and non-compact manifolds. Surprisingly, on  the $N-$dimensional hyperbolic space, which is the model of an unbounded manifold with negative curvature, we find stunning  differences between infinitesimal and large time asymptotics, as opposed to the Euclidean case.\\

Let us outline the results on this article.

In Section 2 we prove our first result that states that for every integrable initial datum there exists a unique solution in $C([0,\infty);L^1(M))\cap L^\infty(0,\infty;L^1(M))$.

We also prove existence and uniqueness as well as a maximum principle for bounded initial data. Then we prove a comparison result for bounded sub- and supersolutions.


In order to study the infinitesimal limit we consider, in Section 3, a family of rescaled operators
\begin{equation*}
 L_\ep u(x):=\frac 1{\ep^2}\int_M \ep^{-N}J\Big(\frac{s_{xy}}\ep\Big)\big(u(y)-u(x)\big)\, d\mu_y
\end{equation*}
with $M$ a spherically symmetric $C^2$ manifold and $J$ with compact support, and we prove that for every $u\in L^\infty(0,T;C^{2+\alpha}_{loc}(M))$ there holds that
\[L_\ep u(x)\to \a\Delta_Mu(x)
 \]
locally uniformly in $M\times[0,T]$ where $\Delta_M$ is the Laplace-Beltrami operator on $M$ and $\a$ is a constant that depends on $J$. When $u\in L^\infty(0,T,C^{2+\alpha}(M))$ the convergence is uniform.

From this result we are able to prove that the solutions $u^\ep$ of the rescaled problems
\[\left\{\begin{split}
&u^\ep_t-L_\ep u^\ep=0\ \ \qquad\mbox{in }M\times(0,T),\\
&u^\ep(x,0)=u_0(x)\qquad\mbox{in }M
\end{split}\right.\]
converge uniformly in $M\times[0,T]$ to the solution $u$ of the Heat-Beltrami  equation
\[\left\{\begin{split}
&u_t-\Delta_Mu=0 \qquad\ \mbox{in }M\times(0,T),\\
&u(x,0)=u_0(x)\qquad\mbox{in }M
\end{split}\right.\]
if $u\in L^\infty(0,T,C^{2+\alpha}(M))$. This regularity is attained if, for instance, $M\in C^3$ and $u_0$ belongs to the closure of $C^3(M)$ in the norm of $C^{2+\alpha}(M)$ (see \cite{simonett}).

We begin the study the large time asymptotics in Section 4 by  looking at the spectral properties of the operator $L$ in the case of compact  manifolds. We prove that there is a nondecreasing sequence of eigenvalues converging to one, the first one being 0.  The normalized eigenfunctions  form an orthonormal basis of $L^2(M)$. We deduce that every solution with initial datum in $L^2(M)$ converges to its average exponentially fast in $L^2(M)$.  For bounded  initial data we prove this convergence in $L^\infty(M)$. We note at this point that in contrast with the usual heat operator, there is no regularizing effect. In particular, if the initial datum is not bounded the solution will not be bounded at any time. Nevertheless, the eigenfunctions are as smooth as the kernel $J$ in the nonlocal operator.

Then, in Section 5 we turn our attention to the case of the $N-$dimensional hyperbolic space as an example of an unbounded,  radially symmetric space. In this case, the infinitesimal limit corresponds to the Heat-Beltrami equation. But, using Fourier transform methods as in \cite{Chasseigne-Chaves-Rossi} for the Euclidean case, we find that, in contrast to the results in \cite{Chasseigne-Chaves-Rossi}, the large time asymptotics for problem \eqref{P} are more related to a translated Heat-Beltrami equation. In fact, the final profile is the same as the one for the problem
\[
\begin{cases}
v_t=(a-1) v+ b\,\Big(\Delta_{\H^N} v+\frac{(N-1)^2}4\,v\Big),\quad& x\in \H^N,\ t>0,\\
v(x,0)=u_0(x),\quad& x\in \H^N.
\end{cases}
\]
Here $a>0$ is a weighted integral of the kernel $J$ related to the radial Fourier transform and $b>0$ is also related to the Fourier transform of $J$.

Finally, there are several interesting features to be discussed. First, it might be more accurate to normalize the kernel $J$ in the case of hyperbolic space in such a way that $a=1$ instead of the usual normalization \eqref{mass=1}. Then, the solutions will behave exactly as those of the translated problem
\[
\begin{cases}
v_t= b\,\Big(\Delta_{\H^N} v+\frac{(N-1)^2}4\,v\Big),\quad& x\in \H^N,\ t>0,\\
v(x,0)=u_0(x),\quad& x\in \H^N,
\end{cases}
\]
that may therefore be seen as a more natural model for diffusion on hyperbolic space than the Heat-Beltrami problem. In particular, in this case the solution decays as $t^{-3/2}$ in any space dimension $N\ge2$ as opposed to the rate $t^{-N/2}$ in Euclidean space, and the exponential decay of the solution of the Heat-Beltrami equation in hyperbolic space.

We include a discussion on these issues at the end of Section 5.

This model has already been proposed as a ``natural" infinitesimal diffusion problem in hyperbolic space in previous articles (see, for instance, \cite{APV:wave,Tataru:Strichartz-hyperbolic}).

\section{Existence, uniqueness and comparison.}
In this section we prove existence, uniqueness and comparison of solutions to the Cauchy problem \eqref{P}. Since the proofs are entirely similar to the ones in $\R^N$ we will only give the main ideas.

To begin with, let us notice that for every $x\in M$ there holds that
\[
 \int_M J(s_{xy})\,d\mu_y=1.\]
In fact, since for every $z\in M$, $s_{xy}=s_{\tau_z(x)\tau_z(y)}$. There holds that $s_{xy}=s_{O\tau_x(y)}$. Now, since $d\mu$ is invariant under the change of variables $z=\tau_x(y)$, there holds that
\[
 \int_M J(s_{xy})\,d\mu_y=\int_M J(s_{Oz})\,d\mu_z=1\quad\mbox{by hypothesis}.\]

Observe that $u$ is a solution to \eqref{P} if and only if
\begin{equation}\label{eqint}
 u(x,t)=e^{-t}u_0(x)+\int_0^t\int_M e^{-(t-r)}J(s_{xy})u(y,r)\,d\mu_y\,dr.
\end{equation}

This is a key point in the proof of existence and uniqueness. There holds,
\begin{theo} Let $u_0\in L^1(M)$. There exists a unique $u\in  C([0,\infty);L^1(M))$ solution to \eqref{P}, and there holds that $\|u(\cdot,t)\|_{L^1(M)}\le \|u_0\|_{L^1(M)}$.

If moreover, $u_0\in L^1(M)\cap L^\infty(M)$ there holds that $\|u(\cdot,t)\|_{L^\infty(M)}\le \|u_0\|_{L^\infty(M)}$.
 \end{theo}
 \begin{proof} It is enough to show that there exists $t_0$ independent of $u_0$ such that the operator
  \begin{equation*}
   \mathcal T v(x,t):=e^{-t}u_0(x)+\int_0^t\int_M e^{-(t-r)}J(s_{xy})v(y,r)\,d\mu_y\,dr
  \end{equation*}
is  a strict contraction in a closed, invariant subset of $C([0,t_0];L^1(M))$. Since $t_0$ is independent of $u_0$ we can continue with the fixed point argument starting from $t_0$. In this way, we get a solution in $C([0,\infty);L^1(M))$.

The proof is very easy and completely similar to the one in $\R^N$ (see \cite{CERW}). In fact, since the set
\[{\mathcal K}=\{v\in C([0,t_0];L^1(M))\,,\,\|v(\cdot,t)\|_{L^1(M)}\le \|u_0\|_{L^1(M)}\}
 \]
 is closed in the norm of $C([0,t_0];L^1(M))$, $  \mathcal T v(x,t)\in {\mathcal K}$ if $v\in {\mathcal K}$ and $\mathcal T$ is a strict contraction in $\mathcal K$ if $t_0$ is small depending only on $\|u_0\|_{L^1(M)}$, there holds that there exists a unique solution in ${\mathcal K}$.

 Analogously, if $u_0\in L^1(M)\cap L^\infty(M)$ and we now let
 \[\widetilde{\mathcal K}=\{v\in {\mathcal K} \,,\,\|v(\cdot,t)\|_{L^\infty(M)}\le \|u_0\|_{L^\infty(M)}\}
 \]
 we again get a closed set invariant under the operator $\mathcal T$, and we get that the unique fixed point in $\mathcal K$ belongs to $\widetilde{\mathcal K}$. Hence, $\|u(\cdot,t)\|_{L^\infty(M)}\le \|u_0\|_{L^\infty(M)}$.
\end{proof}

Now, we get a comparison result between continuous, bounded sub- and supersolutions. There holds,
\begin{prop}\label{comparison}  Assume $J\in L^1(M,(1+s_{xO}^2)\,d\mu)$. Let $u,\ v\in C(M\times[0,T])\cap L^\infty(M\times(0,T))$ be such that
\[\begin{aligned}
u_t-Lu&\le v_t-Lv\quad\mbox{in}\quad M\times(0,T],\\
u(x,0)&\le v(x,0)\ \, \quad\mbox{in}\quad M.
\end{aligned}\] Then, $u\le v$.
\end{prop}
\begin{proof} We divide the proof into two cases.

First, assume $M$ is a compact manifold.

 Let $w=v-u+\delta t$ and assume  that $w$ is negative somewhere in $M\times[0,T]$. Then, there exists $(x_0,t_0)\in M\times(0,T]$ such that $w(y,t)\ge w(x_0,t_0)$ for every $(y,t)\in M\times(0,T]$. Thus,
\[
 0<  w_t(x_0,t_0)-\int_M J(s_{xy})\big(w(y,t_0)-w(x_0,t_0)\big)\,d\mu_y\le 0
\]
which is a contradiction. Therefore, $w\ge0$. Letting $\delta\to0$, we get that $u\le v$.

\smallskip

Now, assume that $M$ is unbounded. Let
\[V(x,t)=4\a Nt+\Big(e^t-\frac12\Big)s^2_{xO}
\]
where $\a=\frac1{2N}\int_MJ(s_{xO})s^2_{xO}\,d\mu_x$.

\smallskip

Then, since $ s^2_{yO}\le 2 s^2_{xO}+2s^2_{xy}$ (this is a consequence of the triangular inequality),
\[
s^2_{yO}-s^2_{xO}\le s^2_{xO}+2s^2_{xy}.
\]
Thus,
\[
\int_MJ(s_{xy})(s^2_{yO}-s^2_{xO})\,d\mu_y\le
s^2_{x0}+2\int_MJ(s_{xy})s^2_{xy}\,d\mu_y=s^2_{xO}+4\a N
\]
because $\int_MJ(s_{xy})s^2_{xy}\,d\mu_y=\int_MJ(s_{zO})s^2_{zO}\,d\mu_z=2 \a N$.

Let now $\ep, \delta>0$ and for $0\le t\le \log\frac32-\ep$, let
\[W(x,t)=w(x,t)+\delta V(x,t).\]
Then,
\[ W_t-LW\ge \delta \Big[4\a N\Big(\frac32-e^t\Big)+\frac12 s^2_{xO}\Big]>0.
\]

Moreover, there exists $R_\delta$ such that, if $s_{x0}\ge R_\delta$, $t\ge0$ there holds that $W(x,t)>0$.
Hence, if $W$ is negative somewhere in $M\times[0,\log\frac32-\ep]$, it attains a minimum at some point $(x_0,t_0)$ with $t_0>0$ and, as in the previous case, we get a contradiction.

Therefore, $W\ge0$ for $0\le t\le  \log\frac32-\ep$.

Letting first $\delta\to0$ and then, $\ep\to0$ we conclude that $u\le v$ for $0\le t\le  \log\frac32$.

We can proceed in a similar way for $\log\frac32\le t\le  2\log\frac32$, etc and we get the inequality $u\le v$ for $0\le t\le T$.
\end{proof}

\begin{rema}\label{remark-comparison} Observe  that the comparison principle holds as long as the operator $L$ can be written as
\[
Lv(x)=c_0\int_M J(s_{xy})\big(u(y)-u(x)\big)\,d\mu_y
\]
independently of the value of $c_0>0$.
\end{rema}

\medskip

We also have existence and uniqueness of  bounded solutions as well as a maximum principle, by assuming only that $u_0\in L^\infty(M)$.

\begin{theo} Let $u_0\in L^\infty(M)$. There exists a unique  solution to \eqref{P} in $C([0,\infty),L^\infty(M))$. Moreover, this solution $u$ satisfies that $\|u\|_{L^\infty(M\times(0,\infty))}\le \|u_0\|_{L^\infty(M)}$.
 \end{theo}
\begin{proof}

It is easy to see that
$$\widehat{\mathcal K}=\{ v\in C([0,t_0],L^\infty(M))\,,\,\|u\|_{L^\infty(M\times(0,\infty))}\le \|u_0\|_{L^\infty(M)}\}$$
is closed under the map $\mathcal T$, and this map is a strict contraction in $\widehat{\mathcal K}$ if $t_0$ is small (independent of $u_0$). Therefore, $\mathcal T$  has a unique fixed point. Iterating this argument, we get a unique global solution.
\end{proof}

\begin{rema} Observe that when $u_0\in C(M)\cap L^\infty(M)$ and $u\in C([0,\infty),L^\infty(M))$ is the unique solution with initial datum $u_0$, there holds that $u\in C(M\times[0,T])$. This can be seen from \eqref{eqint}, for instance.

The same result is true if $u_0\in C(M)\cap L^1(M)$ and $J\in L^1\cap L^\infty$.
\end{rema}

\section{The Laplace Beltrami operator as the infinitesimal limit}

In this section we assume that $J$ has compact support contained in $[0,1]$. Let $M$ be a spherically symmetric $C^2$ manifold which has the property that at any point in $M$, denoted by $O$ there exists a small ball $B_R$ centered at $O$ and polar coordinates $(r,\theta)$ such that the Riemannian metric in $B_R\setminus O$ is given by
$$
ds^2= dr^2+ \psi^2(r)d\theta^2,
$$
where $d\theta^2$ is the standard metric on $\mathbb S^{N-1}$, $\psi\in C^2$,  $\psi(0)=0$ and $\psi'(0)\neq 0$, see \cite{Grigoryan:book}. Without loss of generality we may assume that $\psi'(0)=1$. Here we have made use of the assumption that there exists for every point an isometry $\tau_x$ such that $\tau_x(x)=O$.
\smallskip

Notice
that if $\psi=r$ then $M=\mathbb R^N$, if $\psi = \sin r$ then $M=\mathbb S^N$ and if $\psi=\sinh r$ then $M=\mathbb H^N$.
\smallskip

We consider the measure in $M$ such that, in these coordinates, the volume element assumes the form
$$
d\mu = \psi^{N-1}(r)dr d\theta,
$$
with $d\theta$ the usual area element of the $(N-1)$-dimensional sphere,
and the Laplace Beltrami operator on $M$ is
$$
\Delta_M=  \partial_{rr} + (N-1) \frac{\psi'}{\psi}\partial_r + \frac{1}{\psi^2}\Delta_{\S^{N-1}}
$$
where $\Delta_{\S^{N-1}}$ is the Laplace Beltrami operator on $\mathbb S^{N-1}$. Choose $R$ so small that $B_R \subset M\setminus\{\rm{cut \:locus \:of}\: O\}$. It is easy to see that in this case the unique geodesic from any point $y\in B_R$ to the center $x$ of the ball $B_R$ is a straight line and that the geodesic ball $B_r$ centered at $x$ is
$\{(s,\theta):0<s<r,\:\theta \in \mathbb S^{N-1}\}$.

In this section we consider a family of nonlocal diffusion operators defined on $M$ which are obtained from $L$ by rescaling. Namely,
\begin{equation*}
L_\ep u(x)=\frac1{\ep^2}\int_M\ep^{-N}J\Big(\frac{s_{xy}}\ep\Big)\big(u(y)-u(x)\big)\,d\mu_y.
\end{equation*}
We prove that, when $u\in C^{2+\alpha}_{loc}(M)$, there holds that
\begin{equation}\label{conv-eps}
L_\ep u(x)\to \a\Delta_M u(x)\qquad\mbox{as }\quad\ep\to0
\end{equation}
locally uniformly in $M$ with $\Delta_M$ the Laplace Beltrami operator on $M$. If  $u\in C^{2+\alpha}(M)$, the convergence is uniform and, moreover
\begin{equation}\label{uniform-alpha}\big|L_\ep u(x)- \a\Delta_M u(x)\big|\le C\|u\|_{C^{2+\alpha}(M)}\,\ep^\alpha\quad\forall x\in M.
\end{equation}

\begin{theo} \label{thm-radial} Let $M$ and $L_\ep$ be as above. Let $\a=\frac1{2N}\int_{\R^N} J(|z|)\,|z|^2\,dz$. Then, \eqref{conv-eps} holds locally uniformly in $M$. If $u\in C^{2+\alpha}(M)$ globally, the convergence is uniform in $M$ and \eqref{uniform-alpha} holds.
\end{theo}
\begin{proof}
Let us now introduce polar coordinates in a small ball $B_R$ centered at $x$. In this case $s_{xy}=\rho$ and thus, taking into account that the support of $J$ is contained in $[0,1]$, we obtain for $\ep$ small enough
\begin{align*}
L_\ep u(x) =\ep^{-(N+2)}\int_0^{\ep }J\left(\frac{\rho}{\ep}\right)  \int_{\mathbb S^{N-1}}(u(\rho,\theta)-u(x))\psi^{N-1}(\rho)\,d\theta\,d\rho.
\end{align*}

Since $u\in C^{2+\alpha}_{loc}(M)$ we have,
$$\begin{aligned}
u(\rho, \theta)-u(x)&= u_\rho(\rho,\theta)\rho -\frac{1}{2}u_{\rho\rho}(\rho,\theta)\rho^2 + O(\rho^{2+\alpha})\\
&=u_\rho(\rho,\theta)\rho-\frac12\Delta_M u(\rho,\theta)\rho^2+\frac{N-1}2\frac{\psi'(\rho)\rho}{\psi(\rho)}\rho u_\rho(\rho,\theta)\\
&+\frac12\frac{\rho^2}{\psi^2(\rho)}\Delta_{\S^{N-1}} u(\rho,\theta)+ O(\rho^{2+\alpha}).
\end{aligned}
$$
Observe that $u_\rho(\rho,\theta)$ is the derivative in the direction of  the outer normal to $B_\rho$. Therefore, by applying the divergence theorem on manifolds we get,
$$
\begin{aligned}
&\int_{\mathbb S^{N-1}}(u(\rho ,\theta)-u(x))\psi^{N-1}(\rho)d\theta=\rho\Big(1+\frac{N-1}2\frac{\psi'(\rho)\rho}{\psi(\rho)}\Big)\int_{B_\rho(x)} \Delta_M u(y)\,d\mu_y\\
&-\frac12\int_{\mathbb S^{N-1}}\Delta_Mu(\rho,\theta)\psi^{N-1}(\rho)\rho^2\,d\theta
-\frac{\rho^2}{\psi^2(\rho)}\int_{\mathbb S^{N-1}}\Delta_{\S^{N-1}} u(\rho,\theta) \psi^{N-1}(\rho)\,d\theta+O(\rho^{2+\alpha})\psi^{N-1}(\rho).
\end{aligned}
$$

Now, taking into account that $\int_{\S^{N-1}}\Delta_{\S^{N-1}} u(\rho,\theta)\,d\theta=0$,
$\frac{\psi'(\rho)\rho}{\psi(\rho)}=1+O(\rho)$, $\psi^{N-1}(\rho)=\rho^{N-1}+O(\rho^N)$,
and
$|B_\rho|=\omega_{N-1}(\frac{\rho^N}N+O(\rho^{N+1}))$ with $\omega_{N-1}$ the usual measure of the sphere $\S^{N-1}$ we get,
$$
\begin{aligned}
&\int_{\mathbb S^{N-1}}(u(\rho ,\theta)-u(x))\psi^{N-1}(\rho)d\theta\\
&=\rho\frac{N+1}2\int_{B_\rho(x)}\Delta_M u(y)\,d\mu_y-\frac{\rho^{N+1}}2\int_{\S^{N-1}}\Delta_M u(\rho,\theta)\,d\theta+O(\rho^{N+1+\alpha})\\
&=\rho^{N+1}\Big(\frac{N+1}{2N}-\frac12\Big)\omega_{N-1}\Delta_M u(x)+O(\rho^{N+1+\alpha})
\end{aligned}
$$
Therefore,
\[
 L_\ep u(x)=\frac{\omega_{N-1}}{2N}
\Big[\int_0^\ep J\Big(\frac\rho\ep\Big)\Big(\frac\rho\ep\Big)^2\Big(\frac\rho\ep\Big)^{N-1}\frac{d\rho}\ep\Big]\Delta_M u(x)+O(\ep^\alpha)= \a\Delta_Mu(x)+O(\ep^\alpha).
\]
And the theorem is proved.

Observe that the error term depends only on  $\|\psi\|_{C^2(B)}$ and $\|u\|_{C^{2+\alpha}(B)}$ where $B$ is a small ball around $x$. And, if $u\in C^{2+\alpha}(M)$, \eqref{uniform-alpha} holds with a constant depending only on $\psi$.
\end{proof}

\bigskip

Now we have the elements to prove the convergence of the solution of the $\ep-$problem to the solution of the Heat-Beltrami  equation. There holds,
\begin{theo} Let $u\in C^{2+\alpha,1+\alpha/2}(M\times[0,T])$ be the solution to
 \[
  \begin{cases}
 u_t-\a\Delta_{M} u=0\qquad&\mbox{in }M\times(0,T],\\
u(x,0)=u_0(x)\qquad&\mbox{in }M.
\end{cases}
\]
Let $u^\ep$ be the solution to
\[
  \begin{cases}
 u_t-L_\ep u=0\qquad&\mbox{in }M\times(0,T],\\
u(x,0)=u_0(x)\qquad&\mbox{in }M.
\end{cases}
\]
Then, $u^\ep\to u$ uniformly in $M\times[0,T]$.
\end{theo}
\begin{proof}Let $w^\ep=u-u^\ep$. Then, $\wep$ is a continuous solution to
\[
  \begin{cases}
 w^\ep_t-L_\ep w^\ep=\a\Delta_{M} u-L_\ep u:= F_\ep\qquad&\mbox{in }M\times(0,T),\\
w^\ep(x,0)=0\qquad&\mbox{in }M.
\end{cases}
\]
Recall that we have $F_\ep\to0$ uniformly in $M\times(0,T)$.

Let $z^\ep(x,t)=t\,\|F_\ep\|_{L^\infty(M\times(0,T))}$. Then, $z^\ep$ is the unique bounded solution to
\[
  \begin{cases}
 z^\ep_t-L_\ep z^\ep=\|F_\ep\|_{L^\infty(M\times(0,T))}\qquad&\mbox{in }M\times(0,T),\\
z^\ep(x,0)=0\qquad&\mbox{in }M.
\end{cases}
\]
So, by the comparison principle (Proposition \ref{comparison} and Remark \ref{remark-comparison}) we have
\[
 -z^\ep\le w^\ep\le z^\ep.
\]
Since $ z^\ep\to0$ uniformly in $M\times(0,T)$, the theorem is proved.
\end{proof}

\bigskip

\begin{rema} Let $M$ be a spherically symmetric  $C^3$ manifold and $u_0$ in the clausure of $C^3(M)$ in the norm of $C^{2+\alpha}(M)$. Let $u$ be the unique bounded solution of the Heat-Beltrami equation in $M\times(0,T)$ such that $u\in C([0,T];L^\infty(M))$. Then, $u\in C([0,T]; C^{2+\alpha}(M))$. This is a consequence of the uniqueness of bounded solution and the
existence of solution in $C([0,T];C^{2+\alpha}(M))$ (see \cite{simonett}, Theorem 3.7).
\end{rema}

\section{Spectral properties in compact manifolds.}

In this section we study the spectrum of $L$ when $M$ is  compact and connected. We will assume throughout the section that $J(s)$ is positive for $0\le s<s_0$ for a certain $s_0>0$ and $J$ is Lipschitz in $M$. From the spectral properties we deduce the large time asymptotics of the solutions. We begin by assuming that $u_0\in L^2(M)$. Then, we consider the case of bounded initial data in order to get unifom decay estimates. We recall that $u$ is a solution to the integral equation \eqref{eqint} from which it is clear that if $u_0\in L^1(M)$ is not bounded, $u(\cdot, t)$ will not be bounded for any $t>0$.

\begin{theo} Let us consider $L$ as an operator in $L^2(M)$. There exists a sequence $0=\lambda_0< \lambda_1\le \lambda_2\le\cdots\nearrow 1$ of eigenvalues of $L$ and corresponding eigenfunctions $\{\varphi_k\}$ that form an orthonormal basis in $L^2(M)$.
\end{theo}
\begin{proof} As already observed in \cite{GR} for the case of $\R^N$, we can write $L=L_0-I$ where $I$ is the identity operator and
 \[L_0u(x)=\int_{M} J(s_{xy})u(y)\,d\mu_y.
  \]
We then deduce that $\lambda$ is in the spectrum of -$L$ if an only if $\gamma=1-\lambda$ is in the spectrum of $L_0$.

It is easy to see that $L_0$ is a positive, symmetric operator. Moreover, since $J$ is smooth there holds that $L_0u\in C(M)$ for every $u\in L^2(M)$ and,
\[\|L_0u\|_{L^\infty(M)}\le \|J(s_{x0})\|_{L^2(M)}\|u\|_{L^2(M)},\ \ \| L_0u(x)-L_0u(\bar x)\|_{L^\infty(M)}\le C_{J,M}\|u\|_{L^2(M)}\,s_{x\bar x}.
 \]

Thus, $L_0$ is a compact operator from $L^2(M)$ to $C(M)$ and therefore, also in $L^2(M)$.
Moreover, its spectrum consists of a nonincreasing sequence of positive eigenvalues converging to 0. We conclude that the spectrum of $-L$ consists of a nondecreasing sequence of eigenvalues converging to 1.
It is easy to see that all the eigenvalues are nonnegative. This is, if $-L\varphi=\lambda\varphi$ and $\varphi\not\equiv0$, there holds that $\lambda\ge0$. In fact, by multiplying the equation by $\varphi$ and applying Fubini's theorem we get,
\[
 \begin{aligned}
\lambda\int_{M}\varphi^2\,d\mu&=-\int_{M}\int_{M} J(s_{xy})\big(\varphi(y)-\varphi(x)\big)\varphi(x)\,d\,\mu_y\,d\mu_x\\
&=\frac12\int_{M}\int_{M} J(s_{xy})\big(\varphi(y)-\varphi(x)\big)^2\,d\mu_x\,d\mu_y\ge0.
\end{aligned}
\]
Then we must have $\lambda\ge0$.

Moreover, since the constants are solutions to $Lu=0$, there holds that 0 is an eigenvalue of $L$.
Since $M$ is connected, this eigenvalue is simple. In fact, only the constants are solutions of the homogeneous equation: if $L\varphi=0$ then, $\varphi=L_0\varphi$ is a smooth function. Since $M$ is compact, there exists $x_0\in M$ such that $\varphi(x_0)=\max \varphi$. Let $A$ be the set of points where $\varphi(x)=\varphi(x_0)$. Then, $A$ is a closed, nonempty set. Let us see that it is open. In fact, if $\bar x\in A$,
\[ \varphi(\bar x)=\int_{M}J(s_{\bar xy})\varphi(y)\,d\mu_y\le \varphi(\bar x).
 \]
As a consequence, there is equality and  $\varphi(y)$ must be identically equal to $\varphi(\bar x)$ in the support of $J(s_{\bar xy})$ that contains a neighborhood of $\bar x$ by assumption.

Moreover, by the general spectral theory of positive, symmetric, compact operators in a Hilbert space, there is a set of eigenfunctions of $L_0$ associated to the sequence of eigenvalues $\gamma_k$ that is an orthonormal basis of $L^2(M)$. Since eigenfunctions of $L_0$ associated to $\gamma_k\searrow0$ are eigenfunctions of $-L$ associated to $\lambda_k=1-\gamma_k\nearrow1$, the theorem is proved.
\end{proof}

As a consequence, we have the following result:

\begin{coro} Let $u_0\in L^2(M)$. Let $u$ be the solution to \eqref{P}. Then,
 \[
  \|u(\cdot,t)-\langle u_0\rangle  \|_{L^2(M)}\le e^{-\lambda_1 t}\|u_0\|_{L^2(M)}
 \]
with $\langle u_0\rangle  =\fint_{M}u_0(x)\,d\mu_x$, the mean value of $u_0$ on $M$ and $\lambda_1>0$ the first nonzero eigenvalue of $-L$.

\end{coro}
\begin{proof} By separation of variables we have that
 \[u(x,t)=
  \sum_{k=0}^\infty c_k e^{-\lambda_k t}\varphi_k(x)
 \]
with the sum converging in $L^2(M)$ uniformly in $t$, and $c_k=\int_{M}u_0(x)\varphi_k(x)\,d\mu_x$. Since there holds that $c_0\varphi_0(x)=\langle u_0\rangle  $ and $\lambda_k\ge\lambda_1>0$ for $k\ge1$,
we have the result.
\end{proof}

Now, if moreover $u_0\in L^\infty(M)$ we get the same decay in $L^\infty$ norm. In fact,
\begin{coro} Let $u_0\in L^\infty(M)$. Let $u$ be the solution to \eqref{P}. Then,
 \[
  \|u(\cdot,t)-\langle u_0\rangle  \|_{L^\infty(M)}\le C e^{-\lambda_1 t}\|u_0\|_{L^\infty(M)}
 \]
with $\langle u_0\rangle  =\fint_{M}u_0(x)\,d\mu_x$, the mean value of $u_0$ on $M$ and $\lambda_1>0$ the first nonzero eigenvalue of $-L$.

\end{coro}
\begin{proof} Recall that
\[
 u(x,t)=e^{-t}u_0(x)+\int_0^t\int_M e^{-(t-r)}J(s_{xy})u(y,r)\,d\mu_y\,dr.
\]
Using that $\int_0^te^{-(t-r)}\,dr=1-e^{-t}$, we have
\[
\begin{aligned}
 u(x,t)-\langle u_0\rangle  =& e^{-t}(u_0(x)-\langle u_0\rangle  )+\int_0^t e^{-(t-r)}\int_{M}J(s_{xy})(u(y,r)-\langle u_0\rangle  )\,d\mu_y\,dr\\
&\le Ce^{-t}\|u_0\|_\infty+C_J\|u_0\|_2\int_0^t e^{-(t-r)}e^{-\lambda_1 r}\,dr\\
&=Ce^{-t}\|u_0\|_\infty+C_J\|u_0\|_2\frac{e^{-\lambda_1 t}-e^{-t}}{1-\lambda_1}.
\end{aligned}
\]
Since $1-\lambda_1=\gamma_1>0$, we have that
\[
 |u(x,t)-\langle u_0\rangle  |\le \frac{\bar C}{1-\lambda_1}\big(\|u_0\|_\infty+\|u_0\|_2\big) e^{-\lambda_1 t},
\]
and we immediately obtain the desired conclusion.
\end{proof}

\section{Time asymptotics on hyperbolic space}

In this section we study the large time asymptotics of the solution of \eqref{P} in the case of the hyperbolic space $\H^N$, $N\ge2$.

For the sake of clarity we divide the section into several subsections where we recall the definition and properties of the Fourier transform on hyperbolic space and the study of the solution of translations of the Heat-Beltrami operator, among other results and ideas that are used in the last subsections in order to prove the main result of this section, Theorem \ref{theorem-hyperbolic}.

\subsection{Preliminaries on hyperbolic space and its Fourier transform}

Hyperbolic space $\hh^N$ may be defined as the upper branch of a hyperboloid in $\rr^{N+1}$ with the metric induced by the Lorentzian metric in $\rr^{N+1}$ given by $-dx_0^2+dx_1^2+\ldots+dx_N^2$, i.e., $\hh^N  =\{(x_0,\ldots,x_N)\in \rr^{N+1}: x_0^2-x_1^2-\ldots-x_N^2=1, \; x_0>0\}$, which in polar coordinates may be parameterized as
$$\hh^N= \{x\in \rr^{N+1}: x= (\cosh r, \sinh r \,\theta), \; r\geq 0, \; \theta\in \ms^{N-1}\},$$
with the metric $g_{\mathbb H^N}=dr^2+\sinh^2 r\, d\theta^2,$ where $d\theta^2$ is the canonical metric on $\ms^{N-1}$.
Under these definitions the Laplace-Beltrami operator is given by
$$\Delta_{\hh^N}=\partial_{rr}+(N-1) \frac{\cosh r}{\sinh r}\, \partial_r+\frac{1}{\sinh^2 r}\,\Delta_{\ms^{N-1}},$$
and the volume element is $$d\mu_x=\sinh^{N-1}r\;dr \, d\theta.$$
We denote by $[\cdot, \cdot]$ the internal product induced by the Lorentzian metric
$$[x,x']=x_0x_0'-x_1x_1'-\ldots -x_Nx_N'.$$
The hyperbolic distance between two arbitrary points is given by
$$s_{xx'}=d(x,x')=\cosh^{-1}([x,x']),$$
and in the particular case that $x=(\cosh r, \sinh r \,\omega)$, $x'=O$,
 $$s_{xO}=d(x,O)=r.$$
 The unit sphere $\mathbb S^{N-1}$ is identified with the subset $\{x\in\mathbb R^{N+1} \,:\,[x,x]=0, x_0=1\}$ via the map $b(\omega)=(1,\omega)$ for $\omega\in\mathbb S^{N-1}$.

Finally, note that hyperbolic space may be written as a symmetric space of rank one as the quotient
$\mathbb H^N\approx \frac{SO(1,N)}{SO(N)}.$

Now we start by reviewing some basic facts about the Fourier transform on hyperbolic space, which is a particular case of the Helgason-Fourier transform on symmetric spaces. Some standard references are \cite{Bray,Georgiev,H,Terras:book1,Terras:book2}. First, the generalized eigenfunctions of the Laplace-Beltrami operator may be written as
$$h_{\lambda,\omega}(x)=[x,b(\omega)]^{i\lambda-\frac{N-1}{2}}, \quad x\in \hh^N,$$
where $\lambda\in \rr$ and $\omega \in \ms^{N-1}$. These satisfy
$$\Delta_{\hh^N} h_{\lambda,\omega}=-\left(\lambda^2+\tfrac{(N-1)^2}{4}\right)h_{\lambda,\omega}.$$
In analogy to the Euclidean space, the Fourier transform on $\mathbb H^N$ is defined by
\begin{equation*}\hat{u}(\lambda, \omega)=\int_{\hh^N} u(x)\,h_{\lambda,\omega}(x)\,d\mu_x,
\end{equation*}
for $\lambda\in\rr$, $\omega\in \ms^{N-1}$. Moreover, the following inversion formula holds:
\begin{equation}\label{inversion}
u(x)=\int_{-\infty}^{\infty}\int_{\ms^{N-1}}\bar{h}_{\lambda,\omega}(x)\hat{u}(\lambda,\omega)
\,\frac{d\omega \, d\lambda }{|c(\lambda)|^2},\end{equation}
where $c(\lambda)$ is the Harish-Chandra coefficient:
$$\frac{1}{|c(\lambda)|^2}=\frac{1}{2(2\pi)^N}\frac{|\Gamma(i\lambda+(\frac{N-1}{2})|^2}{|\Gamma(i\lambda)|^2}.$$
There is also a Plancherel formula:
\begin{equation*}
\int_{\hh^N}|u(x)|^2\,d\mu_x=\int_{\rr\times \mathbb S^{N-1}}|\hat{u}(\lambda,\omega)|^2\frac{d\omega \; d\lambda }{|c(\lambda)|^2},\end{equation*}
which implies that the Fourier transform extends to an isometry of $L^2(\mathbb H^N)$ onto $L^2(\mathbb R_+\times\mathbb S^{N-1},|c(\lambda)|^{-2}d\lambda \,d\omega)$.

If $u$ is a radial function, then $\hat u$ is also radial, and the above formulas simplify. In this setting, it is customary to normalize the measure of $\mathbb S^{N-1}$ to one in order not to account for multiplicative constants. Thus one defines the spherical Fourier transform as
\begin{equation*}
\hat u(\lambda)=\int_{\mathbb H^N} u(x)\Phi_{-\lambda}(x)\,d\mu_x
\end{equation*}
where
\begin{equation*}
\Phi_\lambda(x)=\int_{\mathbb S^{N-1}} h_{-\lambda,\omega}(x)\,d\omega
\end{equation*}
is known as the elementary spherical function. In addition, \eqref{inversion} reduces to
\begin{equation*}
u(x)=\int_{-\infty}^{\infty}\hat{u}(\lambda)\Phi_\lambda(x)
\,\frac{d\lambda }{|c(\lambda)|^2}.\end{equation*}
Using polar coordinates on $\mathbb H^N$ (with some abuse in notation since $\Phi_\lambda(x)$ is radial) we have
\begin{equation*}
\Phi_\lambda(r)=c_N\int_0^\pi
(\cosh r-\sinh r\cos\theta)^{i\lambda-\frac{N-1}{2}}(\sin\theta)^{N-2}d\theta.
\end{equation*}
Notice that
\begin{equation*}
\Phi_{-\lambda}(r)=\Phi_\lambda(r)=\Phi_\lambda(-r),\end{equation*}
and that we have normalized $\Phi_\lambda(0)=1$. As a function of $\lambda$, $\Phi_\lambda$ is an even entire function on the whole complex plane, and in particular, we have the expansion:
$$\Phi_\lambda(r)=\Phi_0(r)
+\left.\partial_{\lambda\lambda}\right|_{\lambda=0}\Phi_\lambda(r)\frac{\lambda^2}{2}+o(\lambda^2).$$
Another integral formula is
$$\Phi_\lambda(r)=\frac{c_N}{\sinh^{N-2}r}\int_{-r}^r e^{i\lambda s} (\cosh r-\cosh s)^{\frac{N-1}{2}-1}ds,$$
from where it is easy to estimate
\begin{equation}\label{estimate-symmetric}|\Phi_{\lambda}(r)|\leq |\Phi_0(r)|\leq Ce^{-\frac{N-1}{2}r}(1+r).\end{equation}
For large $|\lambda|$  the estimate may be improved to (see Lemma 2.2 in \cite{Tataru:Strichartz-hyperbolic})
\begin{equation*}
|\Phi_{\lambda}(r)|\leq \frac{C}{|\lambda|^{\frac{N-1}{2}}}\,e^{-\frac{N-1}{2}r} \quad \text{when}\quad|\lambda|\to\infty.
\end{equation*}

Another result in the case of radially symmetric functions is the following (see \cite{Bray}, Thm. 3.3.1 (iii) or \cite{Helgason-Groups-Geometic-Analysis} where this result is actually proven):
if $u\in C_0^\infty(\H^N)$ is radially symmetric, there holds that $\hat u(\lambda)$ can be extended to $\mathbb C$ as an entire function that decays at infinity faster than any negative power of $|\lambda|$.

\medskip

Finally, we define the convolution operator as
$$u*v(x)=\int_{\mathbb H^N} u(x')v(\tau_x^{-1}x')\,d\mu_{x'},$$
where $\tau_x:\mathbb H^N\to \mathbb H^N$ is an isometry that takes $x$ into $O$. If $v$ is a  radial function, then the convolution may be written as
$$u*v(x)=\int_{\mathbb H^N} u(x')v(s_{xx'})\,d\mu_{x'},$$
and we have the property
\begin{equation}\label{product-convolution}\widehat{u*v}=\hat u \,\hat v,\end{equation}
in analogy to the usual Fourier transform.

It is also interesting to observe that
$$\widehat{\Delta_{\mathbb H^N} u}=-\lp \lambda^2+\tfrac{(N-1)^2}{4}\rp \hat u.$$
On hyperbolic space there is a well developed theory of Fourier multipliers. In $L^2$ spaces everything may be written out explicitly. For instance, let $m(\lambda)$ be a multiplier in Fourier variables. A function $\hat u(\lambda,\omega)=\hat m(\lambda) u_0(\lambda,\omega)$, by the inversion formula for the Fourier transform \eqref{inversion} and expression \eqref{k}, may be written as
\begin{equation}\label{inversion1}\begin{split}u(x)=&
\int_{-\infty}^{\infty}\int_{\ms^{N-1}}m(\lambda) \hat{u}_0(\lambda,\omega)\bar{h}_{\lambda,\omega}(x)
\,\frac{d\omega \, d\lambda }{|c(\lambda)|^2}\\
=&\int_{-\infty}^{\infty}\int_{\hh^N}m(\lambda) u_0(x') k_\lambda(x,x')\,
d\mu_{x'} \; d\lambda,
\end{split}\end{equation}
where we have denoted
\begin{equation}\label{k}
k_\lambda(x, x') =\frac{1}{|c(\lambda)|^2}\int_{\ms^{N-1}} \bar{h}_{\lambda,\omega}(x) \,h_{\lambda,\omega}(x')\,d\omega.
\end{equation}
It is known that $k_\lambda$ is invariant under isometries, i.e.,
\begin{equation*}k_\lambda(x, x') =k_\lambda(\tau x, \tau x'),\end{equation*}
for all $\tau\in SO(1,N)$,
and in particular,
$$k_\lambda(x, x') =k_\lambda(s_{xx'}),$$
so many times we will simply write $k_\lambda(\rho)$ for $\rho=s_{xx'}$.
We recall the following formulas for $k_\lambda$ (see, for instance, \cite{Banica:Schrodinger}, which refers to \cite{Georgiev}):

\begin{lemm}\label{lemma-k-lambda} For $N\geq 3$ odd,
\begin{equation} k_\lambda(\rho)=c_N
\left(\frac{\partial_\rho}{\sinh \rho}\right)^\frac{N-1}{2}(\cos\lambda \rho),
\label{Kodd}\end{equation}
and for $N\geq 2$ even,
\begin{equation}\label{Keven}
k_\lambda(\rho)=c_N
\int_\rho^\infty\frac{\sinh s}{\sqrt{\cosh s-\cosh \rho}}
\left(\frac{\partial_s}{\sinh s}\right)^\frac{N}{2}(\cos\lambda s)\, ds.
\end{equation}
\end{lemm}

Before we state our results on hyperbolic space we prove some preliminary statements on the asymptotic behavior of the kernel $k_\lambda$:

\begin{lemm} \label{lemma-derivative1} For each $m=1,2,\ldots$, the derivative
\begin{equation*}
  \lp\frac{\partial_\rho}{\sinh \rho}\rp^{m} \cos (\lambda \rho)
\end{equation*}
is a finite linear combination of terms of the form
\begin{equation*}
  f(\rho^2)g(\lambda,\rho)
\end{equation*}
where $f(z)$ is an analytic function on $|z|\leq\pi^2$, and $g$ can be described in two different ways:
\begin{itemize}
  \item[\emph{i.}] First, 
   \begin{equation*}g(\lambda,\rho)=\lambda^n g_1((\lambda \rho)^2)\end{equation*}
  where $g_1(z)$ is an analytic function on $\mathbb C$, and $n$ is some integer $2\le n\leq 2m$.
  \item[\emph{ii.}] Second, 
  $g$ can be written as a finite linear combination of terms of the form
  \begin{equation*}\lambda^n\frac{\sin(\lambda\rho)}{(\lambda\rho)^l}\quad \text{or}\quad \lambda^n\frac{\cos(\lambda\rho)}{(\lambda\rho)^l},\end{equation*}
  for some nonnegative integers $n,l$, with $2\le n\leq 2m$.
  \end{itemize}
\end{lemm}

\begin{proof}
By induction. For the initial step $m=1$, we calculate
$$\lp\frac{\partial_\rho}{\sinh \rho}\rp \cos (\lambda \rho)=\lambda^2 \lp \frac{\rho}{\sinh \rho} \rp \lp\frac{-\sin (\lambda \rho)}{\lambda \rho}\rp, $$
which is of the specified form. Then, for the step $m+1$, first calculate
$$\frac{\partial_\rho}{\sinh \rho} \left\{f(\rho^2)g_1((\lambda \rho)^2)\right\} =
\frac{2f'(\rho^2)\rho}{\sinh \rho}\,\cdot \,g((\lambda \rho)^2)+\frac{2f(\rho^2)\rho}{\sinh \rho}\,\cdot\,\lambda^2 g'((\lambda \rho)^2),$$
which has the desired form \emph{i}. On the other hand, for \emph{ii.} just notice that
$$\frac{\partial_\rho}{\sinh \rho} \left\{f(\rho^2)\frac{\sin(\lambda\rho)}{(\lambda \rho)^l}\right\} =
\frac{2f'(\rho^2)\rho}{\sinh \rho} \,\cdot\,\frac{\sin(\lambda\rho)}{(\lambda \rho)^l}+\frac{f(\rho^2)\rho}{\sinh \rho}\,\cdot\,\lambda^2\frac{\cos(\lambda\rho)}{(\lambda\rho)^{l+1}}-l\frac{f(\rho^2)\rho}{\sinh \rho}\, \cdot\,\lambda^2\frac{\sin(\lambda\rho)}{(\lambda\rho)^{l+2}},$$
and a similar calculation holds when $\sin\lambda\rho$ is replaced by $\cos\lambda\rho$.

This concludes the proof of the Lemma.
\end{proof}

\begin{lemm}\label{lemma-derivative}
For a function $u=u(\rho)$, and $m=1,2,\ldots$,
\begin{equation*}
  \lp \frac{\partial_\rho}{\sinh \rho}\rp^m u=\sum_{j=1}^{m} h^m_j(\rho) \,\partial_\rho^{(j)} u,
\end{equation*}
such that, for $\rho\geq 1$, we have the bound
\begin{equation*}
     |h^m_j(\rho)| \leq C_j\frac{1}{(\sinh \rho)^{m}},\quad  \quad j=1,\ldots,m.
\end{equation*}
\end{lemm}

\begin{proof}
By induction in $m$ we see that
$h^m_n(\rho)$ is a finite linear combination of terms of the form
\[
\frac{(\cosh\rho)^j}{(\sinh\rho)^{j+m}}.
\]
So, the result follows.
\end{proof}

\begin{coro}\label{cor-bounds}
For every $m=1,2,\ldots$ we have the bound (for all $\lambda\in\mathbb R$)
\begin{equation}\label{bound1}
\left|\lp \frac{\partial_\rho}{\sinh \rho}\rp^m \cos(\lambda\rho)\right|\leq \begin{cases}
C \displaystyle\frac\rho{(\sinh\rho)^{m}}\sum_{n=2}^m\lambda^n,\quad &\rho\geq 1,\\
C \displaystyle\sum_{n=2}^{2m}\lambda^n,\quad  \quad &0<\rho\leq 1.
\end{cases}
\end{equation}
\end{coro}

\begin{proof}
The estimate near $\rho=0$ follows from Lemma \ref{lemma-derivative1}, using expansions \emph{i} or \emph{ii.} according to the value of $\lambda\rho$. On the other hand, for the bound for large $\rho$ we apply Lemma \ref{lemma-derivative} to get
\[
\left|\lp \frac{\partial_\rho}{\sinh \rho}\rp^m \cos(\lambda\rho)\right|\leq Ch_1(\rho)\big|\partial_{\rho}\cos \lambda\rho\big|+\sum_{n=2}^m h_n(\rho)\lambda^{n}
\]
with $h_n(\rho)\le C(\sinh\rho)^{-m}$ for every $n\in\N$. Moreover,
\[
\partial_\rho \cos(\lambda\rho)=-\lambda\sin\lambda\rho=-\lambda^2\rho\,\frac{\sin\lambda\rho}{\lambda\rho}
\]
and the results follows.
\end{proof}

\begin{rema}\label{remark-coro-1.1} Observe that we also have
\[
\left|\lp \frac{\partial_\rho}{\sinh \rho}\rp^m \cos(\lambda\rho)\right|\leq \frac 1{(\sinh\rho)^m}\sum_{n=1}^m \lambda^{n}.
\]
\end{rema}

\begin{lemm}\label{lemma-integral} Let $N\geq 2$ be an even integer.
Assume that the function $h$ satisfies
\begin{equation}\label{equation41}
|h(s)|\leq
\begin{cases}
\sinh^{-\frac N 2} s,\quad &s\geq 1,\\
1, \quad &0<s\leq 1.
\end{cases}
\end{equation}
Then the integral
\begin{equation*}
I(\rho):=\int_\rho^\infty \frac{\sinh s}{\sqrt{\cosh s-\cosh \rho}} \,h(s)\,ds,
\end{equation*}
is bounded by
\begin{equation*}
|I(\rho)|\leq C
\begin{cases}
\sinh^{- \frac{N-1} 2}\rho,\quad &\rho\geq 1,\\
1, \quad &0<\rho\leq 1.
\end{cases}
\end{equation*}

If instead
\begin{equation}\label{cota-h}
|h(s)|\leq
\begin{cases}
s\sinh^{-\frac N 2} s,\quad &s\geq 1,\\
1, \quad &0<s\leq 1.
\end{cases}
\end{equation}
there holds that
\begin{equation*}
|I(\rho)|\leq C
\begin{cases}
\rho\sinh^{- \frac{N-1} 2}\rho,\quad &\rho\geq 1,\\
1, \quad &0<\rho\leq 1.
\end{cases}
\end{equation*}
\end{lemm}

\begin{proof}
 We split the region of integration so that
$$I=\int_\rho^{\rho+2} \frac{\sinh s}{\sqrt{\cosh s-\cosh \rho}} \,h(s)\,ds+\int_{\rho+2}^{\infty} \frac{\sinh s}{\sqrt{\cosh s-\cosh \rho}} \,h(s)\,ds:=I_1+I_2.$$

First assume that $0\leq \rho\leq 1$.
Note that, by Taylor's expansion at the origin,
$$\cosh s-\cosh \rho=\frac{s^2-\rho^2}{2!}+\frac{s^4-\rho^4}{4!}+\ldots,$$
so we can bound  $I_1$ by
$$
\begin{aligned}|I_1|&\leq C\int_\rho^{\rho+2} \frac{\sinh s}{\sqrt{\cosh s-\cosh \rho}} \,ds\leq C\int_\rho^{\rho+2}\frac{s}{\sqrt{s^2-\rho^2}}\,ds,\\
&=C\sqrt{(\rho+2)^2-\rho^2}\le 3C
\end{aligned}
$$

Next, we estimate $I_2$. First assume  $h$ satisfies \eqref{equation41} then, since $0<\rho\le 1$,
\[
\begin{aligned}
|I_2(\rho)|&\le \int_{\rho+2}^\infty\frac{(\sinh s)^{1-\frac N2}}{\sqrt{\cosh s-\cosh\rho}}\,ds\le C\int_2^\infty(\sinh s)^{-\frac{N-1}2}\,ds=C.
\end{aligned}
\]

Now, assume that $\rho\geq 1$. First observe that
\[
\cosh s-\cosh\rho\ge(s-\rho)\sinh\rho.
\]
Therefore,
\[
\begin{aligned}
|I_1(\rho)|&\le \frac1{(\sinh\rho)^{1/2}}\int_\rho^{\rho+2}\frac{(\sinh s)^{1-\frac N2}}{\sqrt{s-\rho}}\,ds\le \frac1{(\sinh\rho)^{\frac{N-1}2}}\int_0^2\tau^{-1/2}\,d\tau=C\frac1{(\sinh\rho)^{\frac{N-1}2}}.
\end{aligned}
\]
In order to bound $I_2(\rho)$, we observe from Taylor's expansion at $s=\rho$ that for $n$ odd and $s>\rho$,
\[
\cosh s-\cosh\rho\ge (\sinh\rho)\frac{(s-\rho)^n}{n!}
\]
so that,
\[
|I_2(\rho)|\le \frac1{(\sinh \rho)^{1/2}}\int_{\rho+2}^\infty\frac{(\sinh s)^{1-\frac N2}}{(s-\rho)^\frac n2}\,ds.
\]
By taking $n>2N$ we get,
\[
|I_2(\rho)|\le \frac C{(\sinh \rho)^{\frac{N-1}2}}.
\]

\medskip

Now, if $h$ satisfies \eqref{cota-h}, we proceed similarly. By using that $s\le \rho+2\le 3$ when $0<\rho\le 1$, the bound of $I_1$ follows in that case. In order to bound $I_2(\rho)$ for $0<\rho\le 1$ we proceed as before and we get,
\[
|I_2(\rho)|\le C\int_2^\infty s(\sinh s)^{-\frac{N-1}2}\,ds=C.
\]

Finally, for $\rho>1$, since $\frac s{s-\rho}\le \frac{\rho+2}2$ for $s\ge \rho+2$,
\[\begin{aligned}
|I_2(\rho)|&\le \frac1{(\sinh \rho)^{1/2}}\int_{\rho+2}^\infty\frac{s(\sinh s)^{1-\frac N2}}{(s-\rho)^{\frac n2}}\,ds\\
&\le C\frac\rho{(\sinh \rho)^{1/2}}
\int_{\rho+2}^\infty\frac{(\sinh s)^{1-\frac N2}}{(s-\rho)^{\frac n2-1}}\,ds.
\end{aligned}
\]
And the result follows by taking $n>2(N+1)$.
\end{proof}

\subsection{The heat equation on hyperbolic space}\label{section-heat}

The explicit expression for the heat kernel on hyperbolic space is well known in the literature  (see, for instance, \cite{Grigoryan-Noguchi} and the references therein, or \cite{Davies-Mandouvalos,Banica:Schrodinger}). However, we provide a direct proof using the Fourier transform in hyperbolic space, since some of the ingredients will be used in the proof of the main theorem.

\begin{prop}\label{prop-heat}
The solution of
\begin{equation}\label{heat}
 \begin{cases}
  v_t=b\Delta_{\mathbb H^N} v+\left(a-1+b\frac{(N-1)^2}{4} \right)v\qquad&\mbox{in }\mathbb H^N\times(0,\infty),\\
v(x,0)=u_0(x)\qquad&\mbox{in }\mathbb H^N,
 \end{cases}
\end{equation}
may be written explicitly  as
\begin{equation*}
v(x,t)=e^{-(1-a)t}\int_{\hh^N}  u_0(x')K_0(s_{xx'},t)\,d\mu_{x'},
\end{equation*}
where, for $N\geq 3$ odd,
\begin{equation}
\label{K0odd}
    K_0(\rho,t)=C_N\frac{1}{\sqrt{bt}}\left(\frac{\partial_\rho}{\sinh \rho}\right)^\frac{N-1}{2}
e^{-\frac{\rho^2}{4bt}},\end{equation}
and, for $N\geq 2$ even,
\begin{equation*}K_0(\rho,t)=C_N\frac{1}{(bt)^{3/2}} \lp\frac{\partial_\rho}{\sinh \rho}\rp^{\frac{N-2}{2}}\int_{\rho}^\infty \frac{s\, e^{-\frac{s^2}{4bt} }}{\sqrt{\cosh s-\cosh \rho}} \,ds.\end{equation*}
Moreover, these kernels satisfy the following estimates when $t\to +\infty$:
\begin{equation}\label{estimates-1}
|K_0(\rho,t)|\leq C\frac{1}{t^{3/2}}
 \frac\rho{(\sinh \rho)^{\frac{N-1}{2}}}, \quad \rho\geq 1,
 \end{equation}
 and,
 \begin{equation}\label{estimates-2}
 |K_0(\rho,t)|\leq C\frac{1}{t^{3/2}},\quad 0<\rho\leq 1,
 \end{equation}

\end{prop}

\begin{proof}
Taking Fourier transform, the solution of \eqref{heat} may be written as
\begin{equation*}
\hat v(\lambda,\omega,t)=\hat u_0(\lambda,\omega)e^{(a-1-b\lambda^2) t}.
\end{equation*}
Therefore, by formula \eqref{inversion1} we can express this solution as
\begin{equation*}
v(x,t)=e^{(a-1)t}\int_{\hh^N}  u_0(x')K_0(s_{xx'},t)\,d\mu_{x'},
\end{equation*}
where
$$K_0(\rho,t)=\int_{-\infty}^{\infty} e^{-b\lambda^2 t} k_\lambda(\rho) \; d\lambda.$$

More precisely, for $N$ odd, using \eqref{Kodd} we can write
\begin{equation}\label{equation50}
K_0(\rho,t)=c\left(\frac{\partial_\rho}{\sinh \rho}\right)^\frac{N-1}{2}\int_{-\infty}^{\infty} e^{-b\lambda^2 t} (\cos\lambda \rho) \; d\lambda.
\end{equation}
But then we calculate
\begin{equation*}
\begin{split}
\int_{-\infty}^{\infty} e^{-b\lambda^2 t} (\cos\lambda \rho) \; d\lambda&=\int_{-\infty}^\infty e^{-b\lambda^2t+i\lambda \rho} \; d\lambda
=\int_{-\infty}^\infty e^{-\lp \lambda\sqrt{bt}-\frac{i\rho}{2\sqrt{bt}}\rp^2} \; d\lambda\\
&=\frac{1}{\sqrt{bt}}\,e^{-\frac{\rho^2}{4bt}}\int_{-\infty}^\infty
e^{-\lp \eta-\frac{i\rho}{2\sqrt{bt}}\rp^2} \; d\eta
=\frac{2\sqrt{\pi}}{\sqrt{bt}}\,e^{-\frac{\rho^2}{4bt}},
\end{split}\end{equation*}
which shows \eqref{K0odd}.

Now we prove the estimates \eqref{estimates-1} and \eqref{estimates-2} in the case that $N$ is odd. First,
we may use Lemma \ref{lemma-derivative1} to bound \eqref{equation50} near $\rho=0$. Indeed, by Lemma \ref{lemma-derivative1}, \emph{i.} if $\lambda\rho\le 1$ and \emph{ii.} if $\lambda\rho>1$ we have, for $0<\rho\le 1$,
\[
\Big|\Big(\frac{\partial_\rho}{\sinh \rho}\Big)^{\frac{N-1}2}\cos\lambda\rho\big|\le C\sum_{n=2}^{{N-1}}\lambda^n
\]
so that, if $t\ge 1$,
\[
|K_0(\rho,t)|\le C\sum_{n=2}^{{N-1}}\int_{-\infty}^\infty \lambda^n e^{-b\lambda^2 t}\,d\lambda=C\sum_{n=2}^{{N-1}}\frac1{(bt)^{\frac{n+1}2}}\int_{-\infty}^\infty
\eta^n e^{-\eta^2}\,d\eta\le \frac C{t^{3/2}}.
\]

Now, for $\rho>1$, by Corollary 1.1,
\[
\Big|\Big(\frac{\partial_\rho}{\sinh \rho}\Big)^{\frac{N-1}2}\cos\lambda\rho\big|\le C\frac{\rho}{(\sinh \rho)^{\frac{N-1}2}}\sum_{n=2}^{\frac{N-1}2}\lambda^n
\]
and the same computation as above gives \eqref{estimates-1}.

On the other hand, for $N$ even
\begin{equation*}K_0(\rho,t)=c\int_{-\infty}^{\infty} e^{-b\lambda^2 t}\int_{\rho}^\infty \frac{\sinh s}{\sqrt{\cosh s-\cosh \rho}} \lp\frac{\partial_s}{\sinh s}\rp^{{N}}\cos \lambda s\,ds\,d\lambda.\end{equation*}

By Lemma 1.2 and Corollary 1.1,  we have,
\[
K_0(\rho,t)=c\sum_{n=2}^{{N}}I_n(\rho)\int_{-\infty}^\infty \lambda^n e^{-b\lambda^2 t}\,d\lambda
\]
where
\[
I_n(\rho)=\int_{\rho}^\infty \frac{\sinh s}{\sqrt{\cosh s-\cosh \rho}}\,h_n(s)\,ds
\]
with
\[
|h_n(s)|\le C\begin{cases}
1\quad& 0<s\le 1,\\
\frac s{(\sinh s)^{\frac N2}}\quad& s>1.
\end{cases}
\]

Hence, by Lemma 1.4 and the computations above, estimates \eqref{estimates-1} and \eqref{estimates-2} hold.
\end{proof}

\begin{rema}\label{rema-initial-decay-heat} Observe from the proof of Proposition 2.1 that for $0<t\le 1$ the power $t^{-3/2}$ in the estimates of $K_0$ is replaced by $t^{-N/2}$, but the $t^{-3/2}$ decay cannot be improved for $t$ large. In the Discussion section at the end we will provide some heuristic arguments explaining this exponent.
\end{rema}

\medskip

\subsection{Asymptotic profile for the non-local equation}

We are given  $J\in C_0^\infty(\H^N)$  a  radially symmetric kernel.
Consider the Cauchy problem
\begin{equation}\label{P5}
 \begin{cases}
  u_t=J*u-u\qquad&\mbox{in }\mathbb H^N\times(0,\infty),\\
u(x,0)=u_0(x)\qquad&\mbox{on }\mathbb H^N.
 \end{cases}
\end{equation}
Recall that the convolution is defined as
$$J*u(x)=\int_{\mathbb H^N} u(x')J(s_{xx'})\,d\mu_{x'}.$$

Note that $\hat J=\hat J(\lambda)$ is a radially symmetric function that has an entire extension to $\mathbb C$  and decays at infinity faster than any negative power of $|\lambda|$. Moreover, the asymptotic behavior  of the solution to \eqref{P5} as $t\to\infty$ is governed by the behavior of  $\hat J(\lambda)$ at the origin. In fact, we have the expansion
\begin{equation*}\hat J(\lambda)=a-b\lambda^2+\lambda^2f(\lambda),\end{equation*}
for some $f(\lambda)=O(\lambda^2)$ an even function (note that $J$ being a radially symmetric kernel has even Fourier transform). These coefficients are given by the formulas
\begin{equation*}
\begin{split}
a&:=\int_{\mathbb H^N} J(s_{xO})\Phi_0(x)\,d\mu_x,\\
b&:=-\frac{1}{2}\int_{\mathbb H^N} J(s_{xO})\left.\partial_{\lambda\lambda}\right|_{\lambda=0}\Phi_\lambda(x)\,d\mu_x.
\end{split}\end{equation*}

Observe that $b>0$. In fact,
\[
\partial_{\lambda\lambda}\Phi_{-\lambda}(\rho)\Big|_{\lambda=0}=
-c_N\int_0^\pi\big(\cosh\rho-\sinh \rho\cos\theta\big)^{-\frac{N-1}2}\log^2
(\cosh\rho-\sinh \rho\cos\theta\big)\,d\theta<0
\]
for every $\rho>0$. Hence,
\[
b=-\partial_{\lambda\lambda}\hat J(0)=\int_0^\infty J(\rho)\partial_{\lambda\lambda}\Phi_{-\lambda}(\rho)\Big|_{\lambda=0}(\sinh\rho)^{N-1}\,d\rho<0.
\]

We will compare to a translated heat equation on hyperbolic space with the same initial condition
\begin{equation}\label{Q}
 \begin{cases}
  v_t=b\Delta_{\mathbb H^N} v+\left(a-1+b\frac{(N-1)^2}{4} \right)v=0\qquad&\mbox{in }\mathbb H^N\times(0,\infty),\\
v(x,0)=u_0(x)\qquad&\mbox{in }\mathbb H^N,
 \end{cases}
\end{equation}
according to Proposition \ref{prop-heat}, we can write explicitly the solution to this equation. Indeed,
$$v(x,t)=e^{-(1-a)t}\int_{\hh^N}  u_0(x')K_0(s_{xx'},t)\,d\mu_{x'},$$
where $K_0$ is given explicitly in the proposition and it satisfies the estimates for $t$ large
\begin{equation*}
\begin{split}
&|K_0(\rho,t)|\leq C\frac{1}{t^{3/2}}
 \frac{\rho}{(\sinh \rho)^{\frac{N-1}{2}}}, \quad \rho\geq 1,\\
&|K_0(\rho,t)|\leq C\frac{1}{t^{3/2}},\quad 0<\rho\leq 1.
\end{split}
 \end{equation*}

Before we state the main result of this section let us prove the following fact:
\begin{lemm} Let $u_0$ be such that the function
\begin{equation*}
v(x):=\sup_{s_{yO}=s_{xO}}|u_0(y)|
\end{equation*}
belongs to $L^1(\H^N)$.
Then, $\hat u_0\in L^\infty(\R\times \S^{N-1})$.
\end{lemm}
\begin{proof}
There holds,
\[\begin{aligned}
|\hat u_0(\lambda,\omega)|&\le \int_0^\infty\int_{\S^{N-1}}|u_0(\cosh\rho,\sigma\,\sinh\rho)||h_{-\lambda,\omega}(
\cosh\rho,\sigma\,\sinh\rho)|\big(\sinh\rho\big)^{N-1}\,d\sigma\,d\rho\\
&\le \int_0^\infty v(\rho)\int_{\S^{N-1}}|h_{\lambda,\omega}(\cosh\rho,\sigma\sinh\rho)\,d\sigma\,(\sinh\rho)^{N-1}\,d\rho.
\end{aligned}
\]
The observation that, by \cite{Tataru:Strichartz-hyperbolic}, Lemma 2.2 (a),
\[\int_{\S^{N-1}}|h_{-\lambda,\omega}(\cosh\rho,\sigma\,\sinh\rho)|\,d\omega=\Phi_0(x)
\le C,
\]
ends the proof.
\end{proof}

Now we state our main result, that is the profile when $t\to +\infty$ of the convolution equation is the same one as the one for the translated Heat-Beltrami equation:

\begin{theo}\label{theorem-hyperbolic} Let $u$ be the solution of \eqref{P5} and $v$ the solution of \eqref{Q}, with initial condition $u_0\in L^\infty(\mathbb H^N)\cap L^1(\mathbb H^N)$ such that $\hat u_0\in L^\infty(\R\times\S^{N-1})$ (observe that this is the case if, for instance, $u_0$ is radially symmetric).

 Then
$$\lim_{t\to\infty} e^{(1-a)t}t^{3/2}\sup_{x\in\mathbb H^N}|u(x,t)-v(x,t)|=0.$$
\end{theo}

\begin{proof}
Taking Fourier transform and using the properties of the convolution on hyperbolic space, the solutions of \eqref{P5} and \eqref{Q} may be written, respectively, as
\begin{equation*}\label{Fourier-u}
\hat u(\lambda,\omega,t)=\hat u_0(\lambda,\omega)e^{(\hat J(\lambda)-1)t}
\end{equation*}
and
\begin{equation*}
\hat v(\lambda,\omega,t)=\hat u_0(\lambda,\omega)e^{(a-1-b\lambda^2) t}.
\end{equation*}
Looking at the expressions for $\hat u$ and $\hat v$, from \eqref{inversion1} we can conclude that
\begin{equation*}
\begin{split}
u(x,t)-v(x,t)&=
\int_{\hh^N} \int_{-\infty}^{\infty} \left[e^{(\hat J(\lambda)-1)t}-e^{(a-1-b\lambda^2)t}\right] k_\lambda(s_{xx'})\,u_0(x')\; d\lambda\,d\mu_{x'}\\
&=e^{(a-1)t}\int_{\hh^N} K(s_{xx'},t) \,u_0(x')\,d\mu_{x'},
\end{split}
\end{equation*}
where
\begin{equation*}
K(\rho,t)=\int_{-\infty}^{\infty} \left[e^{(-b\lambda^2+\lambda^2f(\lambda))t}-e^{-b\lambda^2t}\right] k_\lambda(\rho)\; d\lambda.
\end{equation*}

In order to get the result, we split the integral into two regions: \begin{equation*}
K_1(\rho,t):=\int_{|\lambda|\geq \tau(t)}
\left[e^{(-b\lambda^2+\lambda^2f(\lambda))t}-e^{-b\lambda^2t}\right] k_\lambda(\rho) \, d\lambda
\end{equation*}
and
\begin{equation*}
K_2(\rho,t):=\int_{|\lambda|<\tau(t)}
\left[e^{(-b\lambda^2+\lambda^2f(\lambda))t}-e^{-b\lambda^2t}\right] k_\lambda(\rho) \, d\lambda
\end{equation*}
for some $\tau(t)$ chosen as $\tau(t)=t^{-1/2+\varepsilon}$ for some $\varepsilon$ small enough.
We claim that, both in the even and in the odd case,
\begin{equation}\label{K2}
\begin{split}
&t^{3/2}\sup_{0<\rho\leq 1} \left\{|K_2(\rho,t)|\right\}\to 0, \quad \text{as}\quad t\to +\infty,\\
&t^{3/2}\sup_{\rho\geq 1} \left\{ (\sinh \rho)^{\frac{N-1}{2}}|K_2(\rho,t)|\right\}\to 0, \quad \text{as}\quad t\to +\infty.
\end{split}
\end{equation}

Now we prove this claim. Let us start with $N$ odd. In this case
\[
k_\lambda(\rho)=c\Big(\frac{\partial_\rho}{\sinh \rho}\Big)^{\frac{N-1}2}\cos\lambda\rho
\]
so that, by Lemma 1.2 and Remark 1.1, if $\ep$ is small enough,
\[
\begin{aligned}
|K_2(\rho,t)|&\le\int_{|\lambda|<\tau(t)}e^{-b\lambda^2 t}\big|e^{\lambda^2f(\lambda)t}-1\big|\Big|\Big(\frac{\partial_\rho}{\sinh \rho}\Big)^{\frac{N-1}2}\cos\lambda\rho\Big|\,d\lambda\\
&\le C\int_{|\lambda|<\tau(t)}e^{-b\lambda^2 t}\lambda^2f(\lambda)t\sum_{n=1}^{N-1}\lambda^n\, d\lambda \ \ \begin{cases}
1\quad&0<\rho\le1,\\
\frac1{(\sinh\rho)^{\frac{N-1}2}}\quad&\rho>1
\end{cases}
\end{aligned}
\]
Notice that,
\[\begin{aligned}
t\int_{|\lambda|<\tau(t)}&e^{-b\lambda^2 t}\lambda^{n+2}f(\lambda)\,d\lambda\le Ct\int_{|\lambda|<\tau(t)}e^{-b\lambda^2 t}\lambda^{n+4}\,d\lambda\\
&=\frac {Ct}{(bt)^{\frac {n+5}2}}\int_{|\eta|<\sqrt{bt}\tau(t)}\eta^{n+4}e^{-\eta^2}\,d\eta
\le \frac {C}{t^{\frac {n+3}2}}.
\end{aligned}
\]
Therefore, \eqref{K2} holds.

If $N$ is even we obtain similar bounds by using Lemma \ref{lemma-integral}. In fact,
\[
|K_2(\rho,t)|\le I_n(\rho)\sum_{n=1}^N\int_{|\lambda|<\tau(t)}e^{-b\lambda^2 t}\big|e^{\lambda^2f(\lambda)t}-1\big|\lambda^n\,d\lambda
\]
where
\[
I_n(\rho)=\int_\rho^\infty\frac{\sinh s}{\sqrt{\cosh s-\cosh \rho}}\, h_n(s)\,ds
\]
with
\[
0\le h_n(s)\le c\ \ \begin{cases}1\quad&0<s\le 1,\\
\frac1{(\sinh s)^{\frac{N}2}}\quad&s\ge 1.
\end{cases}
\]

\medskip

Now, in order to estimate
\[
\int_{\hh^N} K_1(s_{xx'},t) \,u_0(x')\,d\mu_{x'}
\]
we split the kernel $K_1$ as $K_1=K_{11}+K_{12}$ with

\begin{equation}\label{K11}
K_{11}(\rho,t)=-\int_{|\lambda|\geq \tau(t)} e^{-b\lambda^2 t} k_\lambda(\rho) \; d\lambda.
\end{equation}
and
\begin{equation*}
K_{12}(\rho,t)=\int_{|\lambda|\geq \tau(t)} e^{(\hat J(\lambda)-a)t} k_\lambda(\rho) \; d\lambda.
\end{equation*}

First we turn our attention to \eqref{K11}. In the case $N$ is odd,
proceeding as before, we get for the interval $0<\rho\le1$, that
\[
|K_{11}(\rho,t)|\le C\sum_{n=2}^{N-1}\int_{|\lambda|>\tau(t)}e^{-b\lambda^2t}\lambda^n\,d\lambda=C\sum_{n=2}^{N-1}
\frac1{(bt)^{\frac{n+1}2}}\int_{|\eta|>b^{1/2}t^\ep}\eta^n e^{-\eta^2}\,d\eta.
\]

So that,
\[t^{3/2}\sup_{0<\rho\le 1}|K_{11}(\rho,t)|\to0\quad\mbox{as}\quad t\to\infty.
\]

Now, if $\rho>1$, by Remark 1.1,
\[\begin{aligned}
|K_{11}(\rho,t)|&\le \frac{C}{(\sinh\rho)^{\frac{N-1}2}}\sum_{n=1}^{\frac{N-1}2}\int_{|\lambda|>\tau(t)}
e^{-b\lambda^2t}\lambda^n\,d\lambda\\
&=\frac{C}{(\sinh\rho)^{\frac{N-1}2}}\sum_{n=1}^{\frac{N-1}2}
\frac1{(bt)^{\frac{n+1}2}}\int_{|\eta|>b^{1/2}t^\ep}\eta^n e^{-\eta^2}\,d\eta\\
&\le \frac{C}{(\sinh\rho)^{\frac{N-1}2}}\ e^{-\frac12b^{1/2}t^\ep}\sum_{n=1}^{\frac{N-1}2}
\frac1{(bt)^{\frac{n+1}2}}\int\eta^n e^{-\eta^2/2}\,d\eta.
\end{aligned}
\]
And we get
\[t^{3/2}\sup_{\rho\ge 1}\ (\sinh\rho)^{\frac{N-1}2}|K_{11}(\rho,t)|\to0\quad\mbox{as}\quad t\to\infty.
\]

\medskip

Finally, we turn to the kernel $K_{12}$. We split it again as
\[
K_{12}=K_{121}+K_{122}
\]
where
\[
K_{121}(\rho,t)=\int_{\tau(t)<|\lambda|\leq R} e^{(\hat J(\lambda)-a)t} k_\lambda(\rho) \; d\lambda
\]
and
\[
K_{122}(\rho,t)=\int_{|\lambda|> R} e^{(\hat J(\lambda)-a)t} k_\lambda(\rho) \; d\lambda.
\]

We choose $R$ small so that, $\hat J(\lambda)-a\le -\tilde b\lambda^2$ if $|\lambda|\le R$ with $\tilde b>0$. Then, since $\hat J$ is analytic, $\hat J(\lambda)\to0$ as $|\lambda|\to\infty$  and $\hat J(\lambda)<a-\delta$ if $|\lambda|=R$ for some positive $\delta$, there holds that $\hat J(\lambda)-a\le -\delta$ if $|\lambda|>R$.

Let us consider the term $K_{121}$. There holds,
$$K_{121}(\rho,t)\leq \int_{\tau(t)<| \lambda|\leq R} e^{-\tilde b\lambda^2 t}
k_\lambda(\rho) \; d\lambda.$$
Proceeding as above we see that \eqref{K2} holds with $K_2$ replaced by $K_{121}$.

We have to treat the term $K_{122}$ in a different way. So, before we proceed let us see what we have up to now.
\[
|u(x,t)-v(x,t)|\le e^{-(1-a)t}\int_{\hh^N} |\tilde K(s_{xx'},t)| \,|u_0(x')|\,d\mu_{x'}
+e^{-(1-a)t}\Big|\int_{\hh^N} K_{122}(s_{xx'},t) \,u_0(x')\,d\mu_{x'}\Big|
\]
where $\tilde K=K-K_{122}$. By the estimates we have already proven,
\[\begin{aligned}
t^{3/2}\int_{\hh^N} &|\tilde K(s_{xx'},t)| \,|u_0(x')|\,d\mu_{x'}\le
o_t(1)\int_{\{s_{xx'}\le 1\}}|u_0(x')|\,d\mu_{x'}\\
&+o_t(1)\int_{\{s_{xx'}\ge 1\}}\frac{|u_0(x')|}{(\sinh s_{xx'})^{\frac{N-1}2}}\,d\mu_{x'}\le o_t(1),
\end{aligned}
\]
under our assumptions for the initial data $u_0$.

In order to finish the proof, let us estimate
\[A=\int_{\hh^N} K_{122}(s_{xx'},t) \,u_0(x')\,d\mu_{x'}.
\]
There holds,
\[\begin{aligned}
A&= \int_{\H^N}\int_{|\lambda|>R}e^{(\hat J(\lambda)-a)t}\int_{\S^{N-1}}h_{\lambda,\omega}(x)h_{-\lambda,\omega}(x')\frac{d\omega}
{|c(\lambda)|^2}\,d\lambda\ u_0(x')\,d\mu_{x'}\\
&=\int_{\S^{N-1}}\int_{|\lambda|>R}e^{(\hat J(\lambda)-a)t}\frac{h_{\lambda,\omega}(x)}{|c(\lambda)|^2}\int_{H^N}h_{-\lambda,\omega}(x')
u_0(x')\,d\mu_{x'},d\lambda\,d\omega\\
&=\int_{\S^{N-1}}\int_{|\lambda|>R}e^{(\hat J(\lambda)-a)t}{h_{\lambda,\omega}(x)}\hat u_0(\lambda,\omega)\frac{d\omega \,d\lambda}{|c(\lambda)|^2}\\
&=e^{-at}\int_{\S^{N-1}}\int_{|\lambda|>R}{h_{\lambda,\omega}(x)}\hat u_0(\lambda,\omega)\frac{d\omega\, d\lambda}{|c(\lambda)|^2}\\
&\qquad+\int_{\S^{N-1}}\int_{|\lambda|>R}e^{-at}\big(e^{\hat J(\lambda)t}-1\big){h_{\lambda,\omega}(x)}\hat u_0(\lambda,\omega)\frac{d\omega \, d\lambda}{|c(\lambda)|^2}=A_1+A_2
\end{aligned}
\]

Using that by hypothesis both $u_0$ and $\hat u_0$ are bounded, the fact that
\begin{equation}\label{formula100}
\int_{\S^{N-1}}|h_{\lambda,\omega}(x)|\,d\omega= \Phi_0(x)\le C\end{equation}
 by \eqref{estimate-symmetric}, the estimate
 $$|c(\lambda)|^{-2}\le C\lambda^2\quad\text{for}\quad\lambda \text{ bounded},$$
 and the formula
\[
\begin{aligned}
A_1=e^{-at}u_0(x)-e^{-at}\int_{\S^{N-1}}\int_{|\lambda|\le R}{h_{\lambda,\omega}(x)}\hat u_0(\lambda,\omega)\frac{d\omega \,d\lambda}{|c(\lambda)|^2},\end{aligned}
\]
it follows that,
\[
t^{3/2}|A_1(x,t)|\le o_t(1)\quad \text{when}\quad t\to +\infty.
\]

\medskip

 On the other hand, using $|c(\lambda)|^{-2}\le C_R|\lambda|^{N-1}$ for $|\lambda|>R$, the same formula \eqref{formula100}, and the fact that $\hat J$ decays faster than any negative power of $|\lambda|$ we get,
\[
\begin{aligned}
|A_2(x,t)|&\le t\,e^{-\delta t}\int_{\S^{N-1}}\int_{|\lambda|>R}|\hat J(\lambda)||h_{\lambda,\omega}(x)||\hat u_0(\lambda,\omega)|\frac{d\omega \, d\lambda}{|c(\lambda)|^2}\\
&\le C t e^{-\delta t}\int_{|\lambda|>R}|\lambda|^{N-1} |\hat J(\lambda)|\,d\lambda\le C t e^{-\delta t}.
\end{aligned}
\]
and the theorem is  proved.
\end{proof}

\subsection{Discussion}

From the previous subsections we find that, in order for the nonlocal problem to behave as the classical translated diffusion equation on hyperbolic space, namely as
\begin{equation}\label{heat-hyperbolic}
v_t=b\,\Big(\Delta_{\H^N}+\textstyle{\frac{(N-1)^2}4}\Big)v
\end{equation}
the kernel $J$ should be normalized so that
\begin{equation}\label{normalization}
a=\int_{\H^N}J(s_{xO})\Phi_0(x)\,d\mu_x=1.
\end{equation}

This normalization differs from the one that preserves mass. This is, to normalize $J$ so that $\int_{\H^N}J(s_{xO})\,d\mu_x=1$ (as a consequence of Fubini's Theorem this normalization implies mass preservation).

In Euclidean space, mass preservation and preservation of $\hat u(0,t)$ are two ways of stating the same fact. Let us show that  the normalization \eqref{normalization} implies that
\[
\int_{\H^N}u(x,t)h_{0,\omega}(x)\,d\mu_x=\hat u(0,\omega,t)
\]
is preserved for every $\omega\in \S^{N-1}$.

In fact,
$$\int_{\mathbb H^N} \big(J*u(\cdot,t)\big)(x)h_{0,\omega}(x)\,d\mu_{x}=\lp\int_{\mathbb H^N} J(s_{xO})h_{0,\omega}(x)\,d\mu_{x}\rp\lp \int_{\mathbb H^N} u(x,t)h_{0,\omega}(x)\,d\mu_{x}\rp.$$
The proof of this identity follows simply by considering the property \eqref{product-convolution} of the convolution on hyperbolic space --using that $J$ is radially symmetric-- and evaluating the Fourier transforms that appear at $\lambda=0$.

As a consequence, if we  impose  the normalization \eqref{normalization}, by multiplying equation \eqref{P5} by $h_{0,\omega}(x)$ and integrating on $\H^N$ we get
\[
\frac d{dt}\int_{\H^N} u(x,t) h_{0,\omega}(x)\,d\mu_x=\int_{\mathbb H^N} \big(J*u(\cdot,t)\big)(x)h_{0,\omega}(x)\,d\mu_{x}-\int_{\mathbb H^N} u(x,t)h_{0,\omega}(x)\,d\mu_{x}=0.
\]
This is, there is conservation of the integral
\begin{equation}\label{conserved1}\int_{\mathbb H^N} u(x,t)h_{0,\omega}(x)\,d\mu_{x}
\end{equation}
for every $\omega\in \S^{N-1}$. In addition, integrating on $\omega\in\mathbb S^{N-1}$ we have that
\begin{equation}\label{conserved2}
\int_{\mathbb H^N} u(x,t)\Phi_0(x)\,d\mu_{x}
\end{equation}
is also preserved.\\

Let us give a closer look at the Heat-Beltrami equation on hyperbolic space. First, observe that \eqref{conserved1} and \eqref{conserved2} are the natural quantities that are preserved for the classical diffusion equation \eqref{heat-hyperbolic} on hyperbolic space  (for the proof, just take Fourier transform and evaluate at $\lambda=0$).

 On the other hand, let us define the quantity
$$P(r,\theta,t)=v(r,\theta,t)\Phi_0(r)(\sinh r)^{N-1}.$$
Observe also that
$$\int_{\mathbb H^{N}} v(x,t)\Phi_0(x)\,d\mu_x=
\int_0^\infty \int_{\S^{N-1}}P(r,\theta,t)\,dr\,d\theta.$$

A straightforward calculation yields, from \eqref{heat-hyperbolic}, the equation
\begin{equation}\label{equation-P}
  P_t=\partial_r \left( P_r -\left[(N-1)\frac{\cosh r}{\sinh r} + 2\partial_r \log \Phi_0\right] P \right)+\frac{1}{\sinh^2 r}\,\Delta_{\mathbb S^{N-1}} P,
\end{equation}
from where we deduce again that $\int_{\mathbb H^{N}} v(x,t)\Phi_0(x)\,d\mu_x$ is constant in time.

  In equation \eqref{equation-P} we can recognize a drift term with velocity
$$V(r)=(N-1)\frac{\cosh r}{\sinh r} + 2\partial_r \log \Phi_0,$$
and, by the asymptotics $\Phi_0(r)=e^{-\frac{N-1}{2}r} (r+o(r))$ as $r\to\infty,$
we find that
$V(r)=\frac{2}{r}+o(r^{-1})$. Since the angular operator in \eqref{equation-P} is multiplied by an exponentially decaying factor, we can expect that the right hand side of \eqref{equation-P} is, at leading order,
\begin{equation*}
  \partial_r \left( P_r+\frac{2}{r}P\right)=\frac{1}{r^2}\partial_r \left( r^2 P_r\right),
\end{equation*}
which is the Euclidean Laplacian for radial functions in dimension 3. It is well known that the decay of the fundamental solution for the corresponding heat equation is $t^{-3/2}$, which motivates the results of the Subsection \ref{section-heat} about the decay of the diffusion on hyperbolic space.
\medskip


\begin{ack}
The authors would like to acknowledge the hospitality of the Newton Institute where part of this work was conducted.
\end{ack}


\medskip

\begin{thebibliography}{xxx}


%

\bibitem{APV:wave} J-P. Anker, V. Pierfelice, M. Vallarino, {The wave equation on hyperbolic spaces}, {\em J. Diff. Eqs.} 252 (2012), 5613--5661.

\bibitem{Banica:Schrodinger}
 V. Banica, The nonlinear Schr\"odinger equation on hyperbolic space. {\em Comm. Partial Differential Equations} 32 (2007), no. 10-12, 1643--1677.

\bibitem{Gonzalez-Saez-Sire} M. Gonz\'alez, M. S\'aez, Y. Sire,
Layer solutions for the fractional Laplacian on hyperbolic space,
{\em Annali di Matematica Pura ed Applicata}, Vol 193, no.6 (2014), 1823--1850.

\bibitem{Banica-Gonzalez-Saez}
V. Banica, M. Gonz\'alez, M. Saez,
Some constructions for the fractional Laplacian on noncompact manifolds, {\em Rev. Mat. Iberoam}. 31 (2015), no. 2, 681--712.


\bibitem{BZ} {P. W. Bates, G. Zhao,}
 Existence, uniqueness and stability of the stationary solution to a nonlocal evolution equation arising in population dispersal. \textit{J. Math. Anal. Appl.} \textbf{332}(1), 428--440 (2007)

\bibitem{Bertoin}
J. Bertoin, {\em L\'evy processes.} Cambridge Tracts in Mathematics, 121. Cambridge University Press, Cambridge, 1996.

\bibitem{Bonforte-Grillo-Vazquez}
M. Bonforte, G. Grillo, J.L. V\'azquez,
Fast Diffusion Flow on Manifolds of Nonpositive Curvature.
{\em J. Evol. Equ.} 8 (2008), no. 1, 99--128.

\bibitem{Bray}
W. Bray, Aspects of harmonic analysis on real hyperbolic space. {\em Fourier analysis (Orono, ME, 1992)}, 77-102, Lecture Notes in Pure and Appl. Math., 157, Dekker, New York, 1994.

\bibitem{CF} {C. Carrillo, P. Fife,} Spatial effects in discrete generation population models.
 \textit{J. Math. Biol.} \textbf{50}(2), 161--188 (2005)


\bibitem{Chasseigne-Chaves-Rossi}
E. Chasseigne, M. Chaves, J. Rossi,
Asymptotic behavior for nonlocal diffusion equations.
{\em J. Math. Pures Appl}. (9) 86 (2006), no. 3, 271--291.

\bibitem{CER} C. Cort\'{a}zar, M. Elgueta, J. D. Rossi, Nonlocal diffusion problems that approximate the heat equation with Dirichlet boundary conditions.  {\em Israel Journal of Mathematics}  170(1), 53-60, (2009)

\bibitem{CERW} {C. Cort\'{a}zar, M. Elgueta, J. D.  Rossi, N. Wolanski,}
How to approximate the heat equation with Neumann boundary conditions by nonlocal diffusion problems. \textit{Arch. Ration. Mech. Anal.} \textbf{187}(1), 137--156 (2008)

\bibitem{CS} L. Caffarelli, L. Silvestre, An extension problem related to the fractional Laplacian, {\em Comm. PDE} 32 (8), 1245-1260.

\bibitem{Davies-Mandouvalos} E. Davies, N. Mandouvalos, Heat kernel bounds on hyperbolic space and Kleinian groups. {\em Proc. London Math. Soc.} (3) 57 (1988), no. 1, 182--208.

\bibitem{D} E. B. Dynkin, Markov Processes, Volume 1,  Die Grundlehren der Mathematischen Wissenschaften Volume 121/122, Springer, 1965.

\bibitem{Fi} {P. Fife,} Some nonclassical trends in
parabolic and parabolic-like evolutions. In ``Trends in nonlinear
analysis", pp. 153--191, Springer-Verlag, Berlin, 2003.

\bibitem{GR} J. Garcia-Melian and J. D. Rossi,   On the principal eigenvalue of some nonlocal diffusion problems. {\em J. of Diff. Equ}. 246(1) (2009), 21--38.

\bibitem{Georgiev} V. Georgiev, {\em Semilinear hyperbolic equations}. MSJ Memoirns, 7. Mathematical Society of Japan, Tokyo, 2000.

 \bibitem{Grigoryan:book} A. Grigor'yan, {\em Heat kernel and analysis on manifolds}. AMS/IP Studies in Advanced Mathematics, 47. American Mathematical Society, Providence, RI; International Press, Boston, MA, 2009.

\bibitem{Grigoryan-Noguchi}
A. Grigor'yan,   M. Noguchi,
\newblock The heat kernel on hyperbolic space.
\newblock {\em  Bulletin of LMS}, 30: (1998), 643--650.


\bibitem{H}
S. Helgason, \newblock {\em Geometric analysis on symmetric spaces. }
 \newblock Second edition. Mathematical Surveys and Monographs, 39. American Mathematical Society, Prnovidence, RI, 2008. xviii+637 pp.

 \bibitem{Helgason-Groups-Geometic-Analysis} \bysame,\newblock {\em Groups and Geometric Analysis}, Academic Press, 1984.

 \bibitem{IR} L. I. Ignat and J. D. Rossi, {\em Refined asymptotic expansions for nonlocal diffusion equations}. J. of Evol. Equ.  8 (2008), 617--629.
%
\bibitem{Saloff-Coste} L. Saloff-Coste,
The heat kernel and its estimates. Probabilistic approach to geometry, 405--436,
{\em Adv. Stud. Pure Math.}, 57, Math. Soc. Japan, Tokyo, 2010.

 \bibitem{simonett} Y. Shao G. Simonett, Continuous maximal regularity on uniformly regular Riemannian manifolds, {\em J. Evol. Equ.} 14 (2014), 211--248.

\bibitem{Tataru:Strichartz-hyperbolic}
D. Tataru,
\newblock Strichartz estimates in the hyperbolic space and global existence for
  the semilinear wave equation.
\newblock {\em Trans. Amer. Math. Soc.}, 353(2):795--807 (electronic), 2001.

\bibitem{TW-critical} J.Terra,  N. Wolanski, Large time  behavior for a nonlocal diffusion equation with
absorption and bounded initial data. \textit{Discrete Contin. Dyn. Syst.} 31 (2), 581--605 (2011).

\bibitem{Terras:book1}
A. Terras,
\newblock {\em Harmonic analysis on symmetric spaces and applications. {I}}.
\newblock Springer-Verlag, New York, 1985.

\bibitem{Terras:book2} \bysame,
{\em Harmonic analysis on symmetric spaces and applications. {II}.} Springer-Verlag, Berlin, 1988.

\bibitem{Vazquez:nonlocal} J.L. V\'azquez,
Nonlinear diffusion with fractional Laplacian operators.  {\em Nonlinear partial differential equations}, 271--298,
Abel Symp., 7, Springer, Heidelberg, 2012.

\bibitem{Vazquez:hyperbolic}
\newblock \bysame,
\newblock Fundamental solution and long time behaviour of the Porous Medium Equation in Hyperbolic Space.
\newblock {\em J. Math. Pur.  Appl.} 104, no. 3, 454--484, 2015.


\end{thebibliography}
\end{document}